\documentclass[a4paper,10pt]{amsart}
\usepackage[utf8]{inputenc}
\usepackage{latexsym, amsfonts, amsthm, amsmath, amssymb, tikz, enumitem, booktabs, float,color,multicol,epsf, url,epsfig,epstopdf, array, setspace, graphicx, csquotes, xcolor,tikz-cd}
\usetikzlibrary{matrix, calc, arrows}

\theoremstyle{plain}
\newtheorem{thm}{Theorem}[section]

\newtheorem{prop}{Proposition}[section]

\newtheorem{lem}{Lemma}[section]

\newtheorem{rmk}{Remark}[section]
\title[Large return time in small contact volume]{Reeb flows with small contact volume and large return time to a global section}
\author{Murat Sa\u glam}
\address{Murat Sa\u glam, Mathematisches Institut,
Universit\"at zu Köln}
\email{msaglam@math.uni-koeln.de}
\begin{document}
\begin{abstract}
We show that any co-oriented closed contact manifold of dimension at least five admits a contact form such that the contact volume is arbitrarily small but the Reeb flow admits a global hypersurface of section with the property that the minimal period on the boundary of the hypersurface and the first return time in the interior of the hypersurface are bounded below. An immediate consequence of this statement is that every co-oriented contact structure on any closed manifold admits a contact form with arbitrarily large systolic ratio. This generalizes the result of Abbondandolo et al. in dimension three to higher dimensions. The proof the main result is inductive and
uses the result of Abbondandolo et al. on large systolic ratio in dimension three in its basis step. The essential construction in the proof relies on the Giroux correspondence in higher dimensions.
\end{abstract}
\maketitle


\section{Introduction and the results}
The main objects of study in this paper are contact manifolds and Reeb flows on them. A \textit{contact structure} $\xi$ is a maximally non-integrable hyperplane distribution on a $(2n+1)$-dimensional manifold $V$. That is, there exists a 1-form $\alpha$ on $V$ such that $\xi={\rm ker}\,\alpha$ and $\alpha\wedge (d\alpha)^n$ is a volume form\footnote{In this manuscript, any contact structure is co-orientable.}. Note that  We call the pair $(V,\xi)$ a \textit{contact manifold} and $\alpha$ a \textit{contact form} on $(V,\xi)$. Note that for any non vanishing function $f$ on $V$, $f\alpha$ is again a contact form on $(V,\xi)$. We say that a contact manifold $(V,\xi)$ is \textit{co-oriented} if $\xi$ is co-oriented. Note that a choice of a contact form $\alpha$ on $(V,\xi)$ determines a co-orientation for $\xi$ via its sign on $TV/\xi$ and an orientation $V$ via the volume form $\alpha\wedge(d\alpha)^n$. A contact form $\alpha$ also defines a natural dynamical system on $V$ via the \textit{Reeb vector field} $R_\alpha$, which is defined by 
$$\imath_{R_\alpha}d\alpha=0\;\;{\rm and}\;\;\imath_{R_\alpha}\alpha=1.$$
Reeb flows naturally arise as Hamiltonian flows on regular energy levels and in particular they generalize the geodesic flows. Consequently the qualitative and quantitative properties of Reeb flows have been important objects of study in dynamical systems as well as contact/symplectic topology.

A classical approach in the study of the smooth flows is reducing the dynamics of the flow to a discrete dynamics on a suitable section of the flow. Given a smooth flow $\phi^t$ on a closed manifold $V$ a \textit{global (hyper)surface}\footnote{Although one is tempted to use the term global surface of \textit{section} in any dimension due to historical reasons, in what follows we will insist on using the term \textit{hypersurface} in order to stress that we are interested in flows in higher dimension.} \textit{of section} for $\phi^t$ is an embedded compact (hyper)surface $F\subset V$ such that
\begin{enumerate}[label=\text{(gs\arabic*)}]
    \item \label{gss1}$\partial F$ is invariant under $\phi^t$,
    \item \label{gss2}$\phi^t$ is transverse to the interior  $F^\circ=F\setminus \partial F$ of $F$,
    \item\label{gss3}for all $p\in V\setminus\partial F$ there exist $t^+>0$ and $t^-<0$ such that $\phi^{t^\pm}(p)\in F^\circ$.
\end{enumerate}
Given such $F\subset V$, one defines the \textit{first return time}
$$\tau:F^\circ\rightarrow (0,+\infty),\;\;\tau(p):={\rm inf}\{t>0\,|\,\phi^t(p)\in F^\circ\}$$
and consequently the \textit{first return map}
$$\Upsilon
:F^\circ\rightarrow F^\circ,\;\;\Upsilon(p)=\phi^{\tau(p)},$$
which encode the dynamics on $V\setminus \partial F$. An immediate consequence in this vein is that the periodic points of $\Upsilon$ correspond to the closed orbits of $\phi^t$ away from $\partial F$ and the periods of these orbits are now bounded below by $\tau$. We note that in general it is hard to obtain such global section for a given flow but once the global section exists it is very handy to study the flow. For a more elaborate account on the existence of global sections and their applications for Reeb flows, we refer to \cite{survey}. 

The main result of this paper is about the existence of contact forms whose Reeb flows admit global hypersurface of sections with very particular properties. 
\begin{thm}\label{maingss}
Let $(V,\xi)$ be a closed co-oriented contact manifold  such that ${\rm dim}V\geq 5$. Then for any $\varepsilon>0$, there exists a contact form $\alpha$ on $(V,\xi)$ with the following properties.
\begin{itemize}
    \item The Reeb flow of $\alpha$ admits a global hypersurface of section $F$ such that the first return time map satisfies $\tau\geq 1 $ on $F^\circ$.
    \item ${\rm T_{min}}(\alpha|_{\partial F})\geq 1$.
        \item ${\rm vol}(\alpha)\leq \varepsilon$.
\end{itemize}
\end{thm}
Here ${\rm vol}(\alpha)$ denotes the \textit{contact volume}, that is the volume of $V$ with respect to the volume form associated to $\alpha$ and ${\rm T_{min}}(\alpha|_{\partial F})$ denotes the the minimum of the periods of all closed orbits of the Reeb flow of the contact form $\alpha|_{\partial F}$ on $\partial F$. Note that indeed $\alpha$ restricts to a contact form on $\partial F$ and the Reeb flow of the restriction coincides with the Reeb flow of $\alpha$ on $\partial F$.

An immediate application of the above theorem is the existence of contact forms with arbitrarily large systolic ratio. In order to motivate this application we briefly discuss the history of the systolic ratio in geometry. 

A classical question in Riemannian geometry is the existence of an upper bound on the length of the shortest non-constant closed geodesic in terms of the Riemannian area on a given closed surface. More specifically on a given closed surface $S$, one studies the functional
\begin{equation}\label{sysratio}
 \rho(S,g)=\frac{{\rm l_{min}}(g)^2}{{\rm area}(g)},   
\end{equation}
on the space of all Riemannian metrics, which is invariant under scaling. Here, ${\rm l_{min}}(g)$ denotes the length of the shortest non-constant closed geodesic and ${\rm area}(g)$ denotes the area of $S$ with respect to the metric $g$. 

In 1949, Loewner showed that if in (\ref{sysratio}), ${\rm l_{min}}( g)$ is replaced with ${\rm sys}_1(g)$, namely with the length of a shortest non-contractible geodesic, the corresponding ratio $\rho_{{\rm nc}}(\mathbb{T}^2,\cdot)$ admits an optimal bound. In 1952, Pu proved the existence of an optimal bound on $\rho_{{\rm nc}}(\mathbb{RP}^2,\cdot)$. In fact, in both statements the metrics that maximize $\rho_{{\rm nc}}$ do not admit any contractible geodesic and hence they also maximize (\ref{sysratio}). In the early 80's, Gromov proved that 
$$\rho_{{\rm nc}}(S,\cdot)\leq 2$$
for any non-simply connected closed surface $S$ but this bound is in general non-optimal  \cite{Gromov}. In fact in \cite{Gromov}, Gromov studied the so called \textit{ systolic ratio} in any dimension and showed that for any essential n-dimensional closed manifold $M$,
$$\rho_{{\rm nc}}(M,g)=\frac{{\rm sys}_1(g)^n}{{\rm vol}(g)}$$
admits an upper bound, which depends only on the dimension. On the other hand, in the late 80's Croke gave the first upper bound on $\rho(S^2\cdot)$ in \cite{croke}, which was later improved by several authors.  

A natural direction for the generalization of the problem is weakening the Riemannian assumption on the metric. In fact,
the ratios $\rho$ and $\rho_{{\rm nc}}$ generalize to the Finsler setting by replacing the Riemannian area with the Holmes-Thompson area and the bounds on $\rho$  generalize to the Finsler case \cite{APBT}. For the detailed account of results about the systolic ratio in Riemannian and Finsler geometry, we refer to \cite{3sphere} and \cite{general}.  

The systolic ratio $\rho$ naturally generalizes to Reeb flows. The \emph{contact systolic ratio} on a closed contact manifold $(V,\xi)$ is defined to be the scaling invariant functional 
$$\rho(V,\alpha):=\frac{{\rm T_{min}}(\alpha)^{n+1}}{{\rm vol}(\alpha)}$$
on the space of all contact forms on $(V,\xi)$. We note that the contact systolic ratio is not merely a generalization of the notion to a particular dynamical system but it is strongly related to the classical definition. In fact, given a smooth Finsler manifold $(M,F)$, the canonical Liouville 1-form $pdq$ on the cotangent bundle $T^*S$, restricts to a contact form $\alpha_F$  on the unit cotangent bundle $S^*_FM$. In this case, the Reeb flow is nothing but the geodesic flow restricted to $S^*_F M$ and up to a universal constant, the contact volume  ${\rm vol}(S^*_F M,\alpha_F)$ agrees with the Holmes-Thompson volume of $(M,F)$. Hence the contact systolic ratio of $(S^*_F M,\alpha_F)$ recovers the classical systolic ratio of $(M,F)$.

But it turns out that it is not possible to bound the contact systolic ratio globally. In the case of the tight 3-sphere $(S^3,\xi_{{\rm st}})$, it was shown in \cite{3sphere} that the systolic ratio can be made arbitrarily large. Yet it was also shown that the \emph{Zoll contact forms}, namely the contact forms for which all Reeb orbits are closed and share the same minimal period, are  maximizers of the functional $\rho(S^3,\cdot)$ if the functional is restricted to a $C^3$-neighbourhood of all Zoll contact forms. For any contact 3-manifold $(M,\xi)$, the non-existence of a global bound on $\rho(M,\cdot)$ is later proved by the same authors in \cite{general} whereas in \cite{Benedetti}, the local bound on $\rho(S^3,\cdot)$ was generalized to all contact 3-manifolds that admit Zoll contact forms.

Going back to our result, we note that a contact form $\alpha$ given by Theorem \ref{maingss} immediately satisfies ${\rm T_{min}}(\alpha)\geq 1$ on $V$ since the period of any closed orbit away from $\partial F$ is bounded below by $\tau$. Hence we prove the following generalization\footnote{Here we need to point that $\rho(V,\alpha)$ makes sense only if the Reeb vector field $R_\alpha$ admits a closed orbit. If $\dim V=3$, by a result of Taubes \cite{Taubes}, we know that any contact form on $V$ admits a closed Reeb orbit but in higher dimensions, this might not be the case. Since we aim for the non-existence of a bound on $\rho$, it is legitimate for us to ignore this issue} of the main result of \cite{general}.
\begin{thm}\label{mainsys}
Let $(V,\xi)$ be a closed co-oriented contact manifold  such that ${\rm dim}V\geq 5$. Then for any $\varepsilon>0$, there exists a contact form $\alpha$ on $(V,\xi)$ satisfying 
$${\rm T_{min}}(\alpha)\geq 1\;\;{\rm and} \;\;{\rm vol}(\alpha)\leq \varepsilon.$$
\end{thm}
The key observation for the proof of Theorem \ref{maingss} is that using the machinery of the Giroux correspondence in higher dimensions, see the next section, one can come up with a careful but rather elementary construction of a contact form that reverses the trivial implication uttered above but now with a shift in dimension. More precisely, we prove the following. 
\begin{prop}\label{keyprop}Let $n\geq2$ be fixed and suppose that Theorem \ref{mainsys} holds for dimension $2n-1$. Then Theorem \ref{maingss} holds for dimension $2n+1$.
\end{prop}
Assuming Proposition \ref{keyprop}, we prove Theorem \ref{maingss} by induction on the dimension. By Proposition \ref{keyprop} and the main result of \cite{general}, which is Theorem \ref{mainsys} in dimension three, we see that Theorem \ref{maingss} holds in dimension five. Now we assume that Theorem \ref{maingss} holds in dimension $2n-1$ for some $n\geq 3$. Then as we observed above Theorem \ref{mainsys} holds in dimension $2n-1$. Proposition \ref{keyprop} implies that Theorem \ref{maingss} holds in dimension $2n+1$ and this finishes the proof.

In \cite{general}, the three dimensional version of Theorem \ref{mainsys} is proven as follows.  On a given closed co-oriented contact three manifold, one constructs a contact form, for which the Reeb flow is Zoll on an invariant domain that occupies arbitrarily large portion of the total contact volume and away from this domain the periods of closed Reeb orbits are bounded away from zero. This construction is carried out on a supported open book decomposition. Then one modifies the Zoll contact form in the large portion with suitable plugs so that the most of the contact volume is eaten up but the minimal period is still bounded away from zero.  The author of this paper provided a proof of Theorem \ref{mainsys} in \cite{msold}, which is a direct generalization of the proof of \cite{general}. It turns out that one can prove the stronger statement given in Theorem \ref{maingss} in a much simpler way and get a proof of Theorem \ref{mainsys} without any plug construction.

The key construction in the proof of Proposition \ref{keyprop} takes place in a  Liouville open book that supports the contact structure symplectically and one uses Theorem \ref{mainsys} to get a contact form on the binding such that contact volume is small and the minimal period is bounded away from zero.  We note that the construction given here does not apply to dimension three since in this case the binding of the open book is one dimensional and the minimal period coincides with the volume. In that sense, the plug construction seems to be essential for the proof of the three dimensional version of Theorem  \ref{mainsys}. We remark that a construction similar to the one in Proposition \ref{keyprop} is used in \cite{entropy} by the authors in order to show the entropy collapse phenomenon in Reeb flows.  
\newline
\textbf{Acknowledgements.} I thank Marcelo Alves, who pointed out that such a proof should work and motivated this paper.  I thank Alberto Abbondandolo  for his comments on this manuscript. I am indebted to the anonymous referee providing insightful comments on the first submitted version of this article and providing directions for additional work which has resulted in the final version of this article. This work is part of a project in the SFB/TRR 191 `Symplectic Structures in Geometry,
Algebra and Dynamics', funded by the DFG.
\section{Liouville open books}
In this section we recall the generalities on Giroux correspondence between the contact structures and supported open books in higher dimensions. For the details, we refer to \cite{openbook} and \cite{ILD}.

Let $F$ be a $2n$-dimensional compact manifold with boundary $K$ and let $F^\circ$ be the interior of $F$. An \emph{ideal Liouville structure}, abbreviated as ILS,  on $F$ is an exact symplectic form $\omega\in \Omega^2(F^\circ)$ with the property that there exists a primitive $\lambda \in \Omega^1(F^\circ)$ and a smooth function such that
\begin{equation}\label{prop_of_u}
u:F\rightarrow [0,+\infty),\;\; {\rm where}\;\; K=u^{-1}(0)\;\; \textrm {is a regular level set,} 
\end{equation}
and the 1-form $u\lambda$ on $F^\circ$ extends to a smooth 1-form on $F$, which is a contact form along $K$. In this case the pair $(F,\omega)$ is called an\emph{ ideal Liouville domain}, shortened to  ILD, and any primitive $\lambda$ of above property is called an \emph{ideal Liouville form}, abbriviated as ILF. It turns out that given an ILF $\lambda$ then the 1-form $u\lambda$ extends to a contact form on $K$ for any function $u$ satisfying \eqref{prop_of_u} but the contact structure 
$$\xi:=\ker (u\lambda)|_{TK}$$
depends only on the 2-form $\omega$, see Proposition 2 in \cite{ILD}. Hence the pair $(K,\xi)$ is called the \emph{ideal contact bounday} of $(F,\omega)$. Moreover, once $\lambda$ is chosen, one can recover all possible (positive) contact forms on $(K,\xi)$ by restricting the extension of $u\lambda$ to $K$ as $u$ moves among the functions with the property (\ref{prop_of_u}). We note that the orientation of $K$ that is determined by the co-oriented contact structure $\xi$ coincides with the orientation of $K$ as the boundary of $(F,\omega)$. 

A very useful feature of an ILD is that a neighbourhood of its boundary can be identified with the symplectization the boundary. 
\begin{lem} \label{lem:nearK} 
Let $(F,\omega)$ be an ILD and $\lambda$ be an ILF. Let $u$ be a function satisfying~\eqref{prop_of_u} 
and let $\beta$ be the extension of~$u\lambda$. Then for any positive contact form $\alpha_0$ on~$(K,\xi)$, 
there exists an embedding 
\begin{eqnarray*}
\imath : [0,+\infty) \times K \rightarrow F
\end{eqnarray*}
such that 
$$
\imath^*\lambda=\frac{1}{r}\alpha_0 \: \textrm{ and } \: \imath(0,q)=q \: \textrm{ for all } \: q\in K ,
$$ 
where $r \in [0,+\infty)$. 
\end{lem} 
The above statement is a reformulation of Proposition~3 in~\cite{ILD} and a proof of this concrete version can be found in \cite{msold}.

Ideal Liouville domains are useful for clarifying the  uniqueness of the contact structures supported by open books in higher dimensions. An \textit{open book} in a closed manifold $V$ is a pair $(K,\Theta)$ where
\begin{enumerate}[label=\text{(ob\arabic*)}]
\item \label{ob1}$K\subset V$ is a closed co-dimension two submanifold with trivial normal bundle;
\item \label{ob2}$\Theta: V\setminus K\rightarrow S^1=\mathbb{R}/2\pi\mathbb{Z}$ is a locally trivial fibration such that $K$ has a neighbourhood $U$, which admits a parametrization $(re^{ix},q)\in \mathbb{D}\times K\cong U$ so that $\Theta$ reads as $\Theta(re^{ix},q)=x$ on $U$.
\end{enumerate}  
Here $K$ is called the \textit{binding} of the open book and the closures of the
fibres of $\Theta$ are called the \textit{pages}. The pages share the common boundary $K$. We note that the canonical orientation of $S^1$ induces co-orientations on the pages and the binding. Hence if $V$ is oriented then so are the pages and the binding. An equivalent way of defining an open book is via a \textit{defining function}. Let $h:V\rightarrow \mathbb{C}$ be a smooth function such that 
\begin{enumerate}[label=\text{(h\arabic*)}]
\item \label{h1}$h$ vanishes transversely;
\item \label{h2}$\Theta:=h/|h|:V\setminus K\rightarrow S^1$ has no critical points, where $K:=h^{-1}(0)$.
\end{enumerate}
Then the pair $(K,\Theta)$ satisfies \ref{ob1} and \ref{ob2}. Conversely any open book in $V$ can be recovered by a function satisfying \ref{h1} and \ref{h2}, which is unique up to multiplication by a positive function on $V$.

Given an open book $(K,\Theta)$ in a closed manifold $V$, one is able to unfold $V$ as follows. One picks  a vector filed $X$,  namely a \emph{spinning vector field}, on $V$ such that 
\begin{enumerate}[label=\text{(m\arabic*)}]
\item \label{m1}$X$ lifts to a smooth vector field on the manifold with boundary obtained from $V$ by a real oriented blow-up along $K$;
\item \label{m2}$X=0$ on $K$ and $(\Theta^*dx)(X)=2\pi$ on $V\setminus K$.
\end{enumerate}
Then the time-one-map of the flow of $X$ is a diffeomorphism 
$$\phi: F\rightarrow F$$ 
of the $0$th-page $F:=\Theta^{-1}(0)\cup K$, which fixes $K$ pointwise. The isotopy class $[\phi]$ of $\phi$ among the diffeomorphisms of $F$ that fixes $K$ pointwise, is called the \emph{monodromy} of the open book and the pair $(F,[\phi])$ characterize the open book as follows. Given $(F,\phi)$, one defines the mapping torus
\begin{equation}\label{eq:MT}
MT(F,\phi):=([0,2\pi]\times F )\big/\sim \,;\; (2\pi,q)\sim (0,\phi(q)),
\end{equation} 
which is a manifold with boundary together with the natural fibration 
$$\widetilde{\Theta}:MT(F,\phi)\rightarrow S^1.$$
Note that all fibres of $\widetilde{\Theta}$ are diffeomorphic to $F$ and there is a natural parametrization of the fibre $\widetilde{\Theta}^{-1}(0)$ via the restriction of the quotient \eqref{eq:MT} to $\{0\}\times F$. It is not hard to see that if $\phi'\in [\phi]$, then there is a diffeomorphism between $MT(F,\phi)$ and $MT(F,\phi')$ that respects the fibrations over $S^1$ and the natural parametrizations of the $0$-th pages. Now given $MT(F,\phi)$, one collapses its boundary, which is diffeomorphic to $S^1\times K$, to $K$ and obtains so called the \emph{abstract open book} $OB(F,\phi)$. Note that the closed manifold $OB(F,\phi)$ admits an open book given by the pair $(K,\Theta) $ where $\Theta$ is induced from $\widetilde{\Theta}$. Moreover, for $\phi'\in [\phi]$, the diffeomorphism between $MT(F,\phi)$ and $MT(F,\phi')$ descends to a diffeomorphism between corresponding abstract open books. Consequently $V$ and $OB(F,\phi)$ may be identified together with their open book structures. We note that one may choose a vector field $X$ that is actually smooth on $V$ (compare with \ref{m1}) and even 1-periodic near $K$ so that the map $\phi={\rm id}$ near $K$ and the boundary of $MT(F,\phi)$ has a neighbourhood of the form $(r_0,0]\times S^1\times K$, which collapses to $r_0\mathbb{D}\times K$ via the map
\begin{equation}\label{eq:collapse}
(r_0,0]\times S^1\times K\rightarrow r_0\mathbb{D}\times K,\;(r,x,q)\mapsto (re^{ix},q).
\end{equation}

The importance of open books in contact topology originates from the following definition. A co-oriented contact structure $\xi$ on a closed manifold $V$ is \emph{supported} by an open book $(K,\Theta)$ on $V$ if there is a contact form $\alpha$ on $(V,\xi)$ such that 
\begin{itemize}
\item $\alpha$ restricts to a (positive) contact form on $K$;
\item $d\alpha$ restricts to a (positive) symplectic form on each fibre of $\Theta$.
\end{itemize}
The celebrated Gioux correspondence tells us that given a closed contact manifold $V$ (isotopy classes of) co-oriented contact structures are in one-to-one correspondence with (equivalence classes of) supporting open books. For our purposes we need to go over certain pieces of this statement in detail. 
\begin{thm}\label{thm10}(Theorem 10 in \cite{openbook}) Any contact structure on a closed manifold is supported by an open book with Weinstein pages.   
\end{thm}
The above statement on the existence is the core part of the correspondence between supporting open books and contact structures. Concerning the uniqueness features of the Giroux correspondence, we are mainly interested in one side, namely the uniqueness of supported contact structures. The main idea is that given an open book, the symplectic geometry of the pages determines the supported contact structures. In dimension three, any two symplectic structure on a page are isotopic since they are simply two area forms on a given surface. Hence the uniqueness statement simplifies a lot. But in higher dimensions, one has to pay attention to the symplectic structure on pages. In \cite{ILD}, Giroux introduced the notion of a Liouville open book, which clears out these technicalities. 

A \emph{Liouville open book}, abbreviated as  LOB, in a closed manifold $V$ is a tripple $\left(K,\Theta,\{\omega_x\}_{x\in S^1}\right)$ where
\begin{enumerate}[label=\text{(lob\arabic*)}]
\item \label{lob1}$(K,\Theta)$ is an open book on $V$ with pages $F_x=\Theta^{-1}(x)\cup K$, $x\in S^1$;
\item \label{lob2} $(F_x,\omega_x)$ is an ILD for all $x\in S^1$ and the following holds:  there is a defining function $h:V\rightarrow \mathbb{C}$ for $(K,\Theta)$ and a $1$-form  $\beta$ on $V$ such that the restriction of $d(\beta/|h|)$ to each page is an ILF. More precisely, 
$$\omega_x=d(\beta/|h|)_{|TF_x^\circ}$$
for all $x\in S^1$.
\end{enumerate}   
The 1-form $\beta$ in \ref{lob2} is called  a \emph{binding 1-form} associated to $h$. Note that if $h'$ is another defining function for $(K,\Theta)$, then  $h'=\kappa h$  for some positive function $\kappa$ on $V$ and $\beta':=\kappa \beta$ is a binding 1-form associated to $h'$. 

Similar to classical open books, LOB's are characterized by the monodromy, which now has to be symplectic. Namely, one  considers a \emph{symplectically spinning vector field}, that is a vector filed $X$ satisfying \ref{m1}-\ref{m2} and generating the kernel of a closed 2-form on $V\setminus K$, which restricts to $\omega_x$ for all $x\in S^1$. Given such a vector field, the time-one-map of its flow, say $\phi$, is a diffeomorphism of $F:=F_0$, which fixes $K$ and preserves $\omega:=\omega_0$. The isotopy class $[\phi]$, among the symplectic diffeomorphisms that fixes $K$, is called the \emph{symplectic monodromy} and characterizes the given LOB. For the construction of a LOB in the abstract open book $OB(F,\phi)$, where $\phi^*\omega=\omega$, we refer to Propostion 17 in \cite{ILD} and our construction in the next section. Similar to spinning vector fields, it is possible to choose a symplectically spinning vector filed whose flow is 1-periodic near the binding. 
\begin{lem}\label{lem15}(Lemma 15 in \cite{ILD}) Let $(K, \Theta,\{\omega_x\}_{x\in S^1})$ be a LOB on a closed manifold $V$ and $h: V\rightarrow \mathbb{C}$ be a defining function for $(K,\Theta)$. Then for every
binding 1-form $\beta$, the vector field $X$ on $V\setminus K$ spanning the kernel of $d(\beta/|h|)$ and
satisfying $(\Theta^*dx)(X)=2\pi$ extends to a smooth vector field on $V$ which is zero along K.
Furthermore, $\beta$ can be chosen so that $X$ is 1-periodic near K.
\end{lem}
As we advertised LOB's arise from contact manifolds.
\begin{prop}\label{prop18} (Proposition 18 in \cite{ILD}) Let $(V, \xi)$ be a closed contact
manifold, and $(K, \Theta)$ be a supporting open book with defining function $h: V\rightarrow \mathbb{C}$.
Then the contact forms $\alpha$ on $(V,\xi)$ such that $d(\alpha/|h|)$ induces an ideal Liouville structure
on each page form a non-empty convex cone.
\end{prop}

The last ingredient of the uniqueness statement is the following definition. Let $(K,\Theta,\{\omega_x\}_{x\in S^1})$ be a LOB on a closed manifold $V$ with a defining function $h$. A co-oriented contact structure $\xi$ on $V$ is said to be  \emph{symplectically supported} by $(K,\Theta,\{\omega_x\}_{x\in S^1})$ if there exists a contact form $\alpha$ on $(V,\xi)$ such that $\alpha$ is a binding 1-form of the LOB associated to $h$. By our remark following the definition of the binding 1-form, the definition of being symplectically supported is independent of the given defining function. But the crucial fact is that once a defining function is fixed, a contact binding 1-form is unique whenever it exists, see Remark 20 in \cite{ILD}. Hence, once a defining function $h$ is fixed, there is a one-to-one correspondence between contact structures supported by $(K,\Theta,\{\omega_x\}_{x\in S^1})$ and contact binding 1-forms associated to $h$. Using this correspondence one shows the following. 
\begin{prop}\label{prop21}(Proposition 21 in \cite{ILD}) On
a closed manifold, contact structures supported by a given Liouville open book
form a non-empty and weakly contractible subset in the space of all contact structures.
\end{prop}
\section{Proof of the key proposition}
This section is dedicated to the proof of Proposition \ref{keyprop}. Let $n\geq 2$ be fixed and
let $(V,\xi)$ be a co-oriented contact manifold of dimension $2n+1$. By Theorem \ref{thm10}, there is an open book $(K,\Theta)$ in $V$ that supports $\xi$. We fix a defining function $h:V\rightarrow \mathbb{C}$ for $(K,\Theta)$. We put $F_x:=\Theta^{-1}(x)\cup K$ for $x\in S^1=\mathbb{R}/2\pi\mathbb{Z}$ and write $F:=F_0$.

By Proposition \ref{prop18}, there is a contact form $\alpha$ 
on $(V,\xi)$ such that $(K,\Theta, d(\alpha/|h|)_{TF_x^o})$ is a LOB, which supports $\xi$ symplectically. By Lemma \ref{lem15}, we modify the contact binding form $\alpha$ only along $\Theta$ and obtain a binding 1-form  $\hat{\alpha}$, not necessarily contact, such that the associated symplectically spinning vector field $X$ is 1-periodic near $K$. 
Hence the time-one-map of the flow of $X$ gives us  a diffeomorphism 
$\psi:F\rightarrow F$ such that  
\begin{equation}\label{monodromy}
\psi^*(d\lambda)=d\lambda
\end{equation}
where $\lambda\in \Omega^1(F^\circ)$ is the ILF given by 
\begin{eqnarray}\label{lambda}
\lambda:=(\hat{\alpha}/|h|)_{|TF^\circ}=(\alpha/|h|)_{|TF^\circ}
\end{eqnarray}
and $\psi={\rm id}$ on some neighbourhood of $K$ in $F$. We want to recover $V$ as the 
abstract open book induced by the pair $(F,\psi)$ and  define a contact form on the abstract open book with the desired properties. To this end we consider the mapping torus $MT(F,\psi)$ and consequently the abstract open book $OB(F,\psi)$. We postpone the precise collapsing procedure for the moment since it would involve choices of certain coordinates, see \eqref{lambdanearK} and \eqref{absOB}. We note the following identifications
$$MT(F^\circ,\psi)=MT(F,\psi)\setminus \partial MT(F,\psi)=OB(F,\psi)\setminus K.$$

\textbf{A family of contact form away from the binding.} 
On $[0,2\pi]\times F^\circ$, we define a 
family of 1-forms 
\begin{eqnarray}\label{alpha_s}
\alpha_s=dx+s\left(\lambda+\kappa(x)\lambda_\psi\right)
\end{eqnarray}
where $\lambda_\psi:=\psi^*\lambda-\lambda$, $s$ is a positive real parameter and $\kappa:[0,2\pi]\rightarrow [0,1]$ is 
a smooth function such that $\kappa(0)=0$, $\kappa(2\pi)=1$ and ${\rm supp}
(\kappa')\subset (0,2\pi)$\footnote{For us ${\rm supp}(f)$ is the closure of the set on which $f$ does not vanish.}. 
By the choice of $\kappa$, $\alpha_s$ descends to a family of 1-forms on  $MT(F^\circ,\psi)$.
We observe the following. 
\begin{lem}\label{contactaway} There exists $s_0>0$, depending on $ \psi,\lambda, \kappa$  such that 
$\alpha_s$ is a contact form on $MT(F^\circ,\psi)$ for all $s\in (0,s_0]$.
\end{lem}
\begin{proof} Since $d\lambda_\psi=0$, we get
$$d\alpha_s
=s\left(\kappa'dx\wedge \lambda_\psi+d\lambda\right)$$
$$\Rightarrow\;(d\alpha_s)^n=s^n\left((n-1)\kappa'dx\wedge \lambda_\psi\wedge (d\lambda)^{n-1}+(d\lambda)^n\right)$$
and 
$$\alpha_s\wedge (d\alpha_s)^{n}
=\left[dx+s\left(\lambda+\kappa\lambda_\psi\right)\right]\wedge s^n\left[(n-1)\kappa'dx\wedge \lambda_\psi\wedge (d\lambda)^{n-1}+(d\lambda)^n\right]$$
$$\Rightarrow\;\frac{\alpha_s\wedge (d\alpha_s)^{n}}{s^n}=dx\wedge (d\lambda)^n+s\kappa'\lambda \wedge dx\wedge \lambda_\psi\wedge (d\lambda)^{n-1}.$$
Note that $dx\wedge (d\lambda)^n$ is a volume form and  the top degree form $\lambda \wedge dx\wedge \lambda_\psi\wedge (d\lambda)^{n-1}$ is compactly supported in $MT(F^\circ,\psi)$.  Hence there exists $s_0>0$ such that the right hand side of the above equation is positive volume form for all $s\in(0,s_0]$.
\end{proof}

Now we want to understand the Reeb vector field $R_{\alpha_s}$ of $\alpha_s$ on $MT(F^\circ,\psi)$. To this end we define 
the vector field $Y$ on $F$ via 
$$\imath_{Y}d\lambda=\lambda_\psi.$$
Note that $Y$ is compactly supported and $\kappa'Y$ is a well-defined compactly supported vector field on $MT(F^\circ,\psi)$ since $\kappa'$ is compactly supported in $(0,2\pi)$. We compute 
\begin{eqnarray*}
\imath_{(\partial_x-\kappa'Y)} d \alpha_s
 &=& s  \left( \imath_{(\partial_x-\kappa'Y)} \kappa' dx \wedge \lambda_\psi + \imath_{(\partial_x-\kappa'Y)} d \lambda \right) \\
 &=& s \, \kappa' \left( \lambda_\psi +\kappa'\lambda_\psi (Y) \, dx\right) - s\kappa'\imath_{Y} d \lambda \\
 &=&s\kappa'\lambda_\psi-s\kappa'\lambda_\psi \\
 &=&0
\end{eqnarray*}
and 
$$\alpha_s(\partial_x-\kappa'Y)=1-s\kappa'\lambda(Y).$$
Hence on $MT(F^\circ,\psi)$, the Reeb vector field of $\alpha_s$ reads as
\begin{eqnarray}\label{ReebsonW'comp}
 R_{\alpha_s}=\frac{\partial_x-\kappa'Y}{1-s\kappa'\lambda(Y)}.
\end{eqnarray}
We note that  $R_{\alpha_s}=\partial_x$ near $K$. Since the $\partial_x$ component of $R_{\alpha_s}$  never vanishes and  $\kappa Y$ is tangent to the pages, $R_{\alpha_s}$ is transverse to $F^\circ\times \{x\}$ for all $x$. Hence $F^\circ$ is a global hypersurface of sections for $R_{\alpha_s}$ on  $MT(F^\circ,\psi)$ in the sense that $F^\circ$ is a properly embedded hypersurface in $MT(F^\circ,\psi)$ and the properties given by \ref{gss2} and \ref{gss3} hold. Consequently we have the first return time   
\begin{eqnarray}\label{returntimeaway}
\tau_s:F^\circ\rightarrow \mathbb{R},\; \;\tau_s(p)=\inf \{t>0\;|\; \phi_{R_{\alpha_s}}^t(0,p)\in \{0\}\times F^\circ\}
\end{eqnarray}
and the first return map 
\begin{eqnarray}\label{returnmapaway}
\Upsilon: F^\circ\rightarrow F^\circ;\;\; (0,\Upsilon(p))=\phi_{R_{\alpha_s}}^{\tau_s(p)}(0,p), \;\forall p\in F^\circ.
\end{eqnarray}
\begin{rmk}\label{returnmapindependent}
We note that since $R_{\alpha_s}$ is multiple of the vector field $\partial_x-\kappa'Y$ and latter is independent of $s$. Hence the return map $\Upsilon$ is independent of $s$, which justifies the absence of the subscript in (\ref{returnmapaway}).
\end{rmk}
We note that for all $s\in(0,s_0]$
\begin{eqnarray}\label{tau_sh_snearK}
\tau_s\equiv 2\pi,\;\;\Upsilon={\rm id}\;\;{\rm on}\;\; F\setminus {\rm supp}\,(\psi).
\end{eqnarray}
\begin{lem}\label{tauboundedawayK} There exists $s_1<s_0$ depending on $\psi,\lambda$ and $\kappa$ such that for all $s\in(0,s_1]$,
\begin{eqnarray}
 \pi\leq \tau_s\leq 4\pi
\end{eqnarray}
on $F^\circ$.
\end{lem}
\begin{proof} We have
$$dx(R_{\alpha_s})=\frac{1}{1-s\kappa'\lambda(Y)},$$
which converges to 1 uniformly as  $s\rightarrow 0$. This follows from the fact that $Y$ is compactly supported. Then $\tau_s$ converges uniformly to $2\pi$ and the statement follows.
\end{proof}
\textbf{A family of contact forms near the binding.}
By our assumption and after a suitable re-scaling, we know that for any $c>0$ there exists a contact form $\sigma_c$ on $(K, \xi|_K)$ such that 
\begin{eqnarray}\label{sigma_epsilon}
{\rm vol}(K,\sigma_c)\leq c\;\;{\rm  and}\;\; T_{{\rm min }}(\sigma_c)\geq \pi.
\end{eqnarray}
Given $\sigma_c$, by Lemma \ref{lem:nearK} there is an embedding 
\begin{eqnarray}\label{lambdanearK}
\imath_c:[0,+\infty)\times K\hookrightarrow F\;\;{\rm s.t.}\;\;\imath_c^* \lambda =\frac{1}{r}\sigma_c
\end{eqnarray}
and there exists $r_c>0$ depending only on $\psi$ and $\sigma_c$ such that 
\begin{equation}\label{psi=id}
\imath_c \left([0,r_c]\times K\right)\cap \,{\rm supp}(\psi)=\emptyset.
\end{equation}
We define
\begin{eqnarray}\label{F_eepsilon}
F_c:=F\setminus \imath_c\left([0,r_c)\times K\right)
\end{eqnarray}
and note that near the boundary of $MT(F_c,\psi)$, (\ref{alpha_s}) reads as
\begin{eqnarray}\label{alpha_snearendofF_epsilon}
\alpha_s=dx+\frac{s}{r}\sigma_c.
\end{eqnarray}
We now want to extend the coefficients of \eqref{alpha_snearendofF_epsilon} to a neighbourhood of $K$ in $OB(F,\psi)$.
\begin{lem} \label{fandg} 
Given $a\in (0,1/2)$, there exist smooth functions 
$$
f,g : [0,1] \rightarrow [0,1]
$$ 
with the following properties. 
\end{lem}

\begin{enumerate}[label=\mbox{(f\arabic*)}]
\item \label{f1}  {\it $f(y)=a/y$ near $y=1$.} 
\smallskip
\item \label{f2}  {\it  $f(r)=1-y/2$ near $y=0$.}
\smallskip
\item \label{f3} $f(1/2)=1/2$. 
\smallskip
\item \label{f'}  $-2 \leq f'<0 $ on $(0,1]$.
\end{enumerate}
\begin{enumerate}[label=\mbox{(g\arabic*)}]
\item \label{g1}   {\it $g=1$ near $y=1$.}

\smallskip
\item \label{g2}   {\it  $g(y)=y^2/2$ near $y=0$.}

\smallskip
\item\label{g3} $g(1/2)=1/2$

\smallskip
\item \label{g'}  {\it $0\leq g'\leq 2$ on $[0,1]$ and $0<g'$ on $(0,1/2]$.} 
\end{enumerate}
\begin{enumerate}[label=\mbox{(fg)}]
\item\label{fg} {\it $g'/f'\leq -1$ on $(0,1/2]$.}
\end{enumerate}
\noindent 
We leave the proof the lemma to the reader\footnote{In fact one can take $f=1-g$ on $[0,1/2]$ and get $g'/f'=-1$ in \ref{fg}.}. 

Given $r_c$ as above and any $s\in(0,r_c/2)$, we apply the above lemma for $a=s/r_c$ and get the functions $f$ and $g$. We define
1-form 
\begin{eqnarray}\label{alpha_sepsnearK}
\alpha_{s,c} (x,r,q) = g_{s,c}(r)\, dx + f_{s,c}(r)\,  \sigma_c (q)
\end{eqnarray}
on $[0,r_c] \times S^1 \times K$, where
\begin{eqnarray}\label{f_sc}
f_{s,c}:[0,r_c]\rightarrow [0,1],\;\; f_{s,c}(r):=f(r/r_c),
\end{eqnarray}
and
\begin{eqnarray}\label{g_sc}
g_{s,c}:[0,r_c]\rightarrow [0,1],\;\; g_{s,c}(r):=g(r/r_c).
\end{eqnarray}
We note that by \ref{f2} and \ref{g2}
$$\alpha_{s,c}=\frac{r^2}{2r_c^2}dx+\big(1-\frac{r^2}{2r_c^2}\big)\sigma_c$$
near $r=0$ and therefore $\alpha_{s, c}$ descends to a  smooth 1-form on $ r_c \mathbb{D} \times K$.
\begin{lem} \label{contactnear} 
For the above choices, $\alpha_{s, c}$ is a contact form on $r_c \mathbb{D} \times K$.
\end{lem}

\begin{proof} 
We compute
\begin{eqnarray*}
\alpha_{s, c} \wedge (d\alpha_{s, c})^n
&=& (g_{s, c} \, dx + f_{s, c}  \,\sigma_ c) \wedge (g'_{s, c} \,dr \wedge dx + f_{s, c}'\, dr \wedge \sigma_ c + f_{s, c} \, d\sigma_ c)^n 
     \notag \\
&=&n\,h_{s, c}\, f_{s, c}^{n-1} \big( dr \wedge dx \wedge \sigma_ c \wedge (d \sigma_ c)^{n-1} \big), 
\end{eqnarray*}
where
\begin{eqnarray}\label{h_sc}
h_{s, c}:=f_{s, c}g'_{s, c}-f'_{s, c}g_{s, c}.
\end{eqnarray}
Note that by \ref{f1} and \ref{f'}, $f_{s, c}>0$ on $[0,r_c]$ and by \ref{f'}, \ref{g2} and \ref{g'}
$$h_{s, c}\geq -f'_{s, c}g_{s, c}>0$$
on $(0,r_c]$. Therefore, $\alpha_{s, c}$ is a contact form away from~$K$. 
Near $K$ we have 
$$h_{s, c}=\frac{r}{r_c^2}$$ and hence $\alpha_{s, c} \wedge (d\alpha_{s, c})^{n}$ reads 
$$
\frac{n}{r_c^2}\,\left(1-\frac{r^2}{2r_c^2}\right)^{n-1}\bigl( r dr \wedge dx \wedge \sigma_ c \wedge (d \sigma_ c)^{n-1} \bigr) ,
$$
which is a positive volume form at any point on $K$.
\end{proof}

An easy computation shows that away from K, the Reeb vector field $R_{s, c}$ of $\alpha_{s, c}$ reads as
\begin{equation}\label{reebnearK}
R_{s, c}(x,r,q)=-\frac{f_{s, c}'}{h_{s, c}}\partial_x+\frac{g_{s, c}'}{h_{s, c}}R_{c}(q)    
\end{equation}
and has the flow
\begin{equation}\label{reebflownearK}
\phi^t_{s, c}(x,r,q)=\left( x-\frac{f_{s, c}'(r)}{h_{s, c}(r)}t,\,r,\, \phi_{ c}^{\frac{g_{s, c}'(r)t}{
h_{s, c}(r)}}(q)\right),    
\end{equation}
where $\phi_{c}^t$ is the flow of the Reeb vector field $R_{c}$ of $\sigma_c$ on $K$.
For $(0,q)\in r_c\mathbb{D}\times K$, we have
\begin{equation}\label{reebandflowonK}
R_{s, c}(0,q)=R_{c}(q),\;\phi^t_{R_{s, c}}(0,q)=\left(0,\phi_{c}^{t}(q)\right).
\end{equation}
\\
\textbf{A family of contact forms on $OB(F,\psi)$.}
Given $ c>0$ and $s\in(0,\min\{s_1,r_ c/2\})$, we define, abusing the notation,
\begin{eqnarray}\label{alpha_epseps}
\alpha_{s, c}=   \left\{
\begin{array}{ll}
      \alpha_s &\textrm{ on } MT(F_{ c},\psi)\\
      \\
      g_{s,c}dx+f_{s,c}\sigma_ c &\textrm{ on } r_ c\mathbb{D}\times K\\
\end{array} 
\right.
\end{eqnarray}
on the abstract open book
\begin{equation}\label{absOB}
OB(F,\psi)=MT(F_{ c},\psi)\cup \left(r_ c\mathbb{D}\times K\right)
\end{equation} 
where $\alpha_s$ is defined by (\ref{alpha_s}) and $f_{s,c}$ and $g_{s,c}$ are given by \eqref{f_sc} and \eqref{g_sc}. By (\ref{alpha_snearendofF_epsilon}) and the properties \ref{f1} and \ref{g1}, $\alpha_{s, c}$ is a well-defined contact form on $OB(F,\psi)$. 

By the construction of $\alpha_s$ and \ref{f'}, $F$ is a global surface of section for the Reeb flow of $\alpha_{s,c}$ with the first return time
\begin{eqnarray}\label{tau_sc}
\tau_{s,c}:F^\circ\rightarrow (0,\infty).
\end{eqnarray}
We note that $\tau_{s,c}=\tau_s$ on $F_c$ and since $s<s_1$ Lemma \ref{tauboundedawayK} implies that 
\begin{eqnarray}\label{tau_scboundawayK}
\tau_{s,c}\geq \pi \;\;{\rm on }\;\; F_c.
\end{eqnarray}
By \eqref{reebflownearK} we have
$$\tau_{s,c}(r,q)=-2\pi\,\frac{h_{s,c}(r)}{f'_{s,c}(r)}=2\pi\left(g_{s,c}(r)-f_{s,c}(r)\,\frac{g'_{s,c}(r)}{f'_{s,c}(r)}\right)$$
on $(0,r_c)\times K$. On $(0,r_c/2)\times K$ we have
$$\tau_{s,c}\geq 2\pi\,f_{s,c}\left(-\frac{g'_{s,c}}{f'_{s,c}}\right)\geq \pi\left(-\frac{g'_{s,c}}{f'_{s,c}}\right)\geq \pi$$
where the second inequality follows from \ref{f3} and \ref{f'} and the last inequality follows from \ref{fg}. On $[r_c/2,r_c)\times K$ we have
$$\tau_{s,c}\geq 2\pi\,g_{s,c}\geq \pi$$
where the last inequality follows from \ref{g3} and \ref{g'}. Combining these observations with \eqref{tau_scboundawayK} we conclude that
\begin{eqnarray}\label{tau_scbounded}
\tau_{s,c}\geq \pi\;\; {\rm on}\;\; F^\circ.
\end{eqnarray}
Moreover by \eqref{reebandflowonK} we also get 
\begin{eqnarray}\label{Tminbounded}
{\rm T_{min}}(\alpha_{s,c}|K)={\rm T_{min}}(\sigma_c)\geq \pi.
\end{eqnarray}

We estimate the volume of $OB(F,\psi)$ with respect to $\alpha_{s, c}$ as follows. 
$$\int_{OB(F,\psi)}\alpha_{s, c}\wedge (d\alpha_{s, c})^n=\int_{MT(F_{ c},\psi)}\alpha_{s}\wedge (d\alpha_{s})^n+
\int_{r_ c\mathbb{D}\times K}\alpha_{s, c}\wedge (d\alpha_{s, c})^n.$$
Since $s<s_1$, we have 
$$\int_{MT(F_ c,\psi)}\alpha_{s}\wedge (d\alpha_{s})^n=\int_{F_{ c}}\tau_s\,(d\alpha_s|_{\{0\}\times F_ c})^n=\int_{F_{ c}}\tau_s\,s^n(d\lambda)^n\leq 4\pi s^n\int_{F_{ c}}(d\lambda)^n,$$
where the inequality follows from Lemma \ref{tauboundedawayK}. For the second term we have
\begin{eqnarray*}
\int_{r_ c\mathbb{D}\times K}\alpha_{s, c}\wedge (d\alpha_{s, c})^n
&=& \int_{r_ c\mathbb{D}\times K}n h_{s, c}f_{s, c}^{n-1}\big( dr\wedge dx\wedge\sigma_ c \wedge (d\sigma_ c)^{n-1}\big)\\
&=&2\pi n\,{\rm vol}(K,\sigma_ c)\int_0^{r_ c} h_{s, c}f_{s, c}^{n-1} dr\\
&\leq & 2\pi n c \int_0^{r_ c} h_{s, c} dr\\
&\leq & 2\pi n c \int_0^{r_ c} \frac{4}{r_ c} dr\\
&=&8\pi n c. 
\end{eqnarray*}
where the last inequality follows from \ref{f'}, \ref{g'} and   
$$h_{s, c}\leq g'_{s, c}-f'_{s, c}.$$
Hence for $s\in(0,\min\{s_1,r_ c/2\})$ we get
\begin{equation}\label{volumebound}
   {\rm vol}(\alpha_{s, c})\leq 4\pi s^n\int_{F_{ c}}(d\lambda)^n+8\pi n c. 
\end{equation}
Recall that $s_1$ depends only on $\lambda, \psi$ and $\kappa$. Now given any $\varepsilon>0$, we choose $c>0$ such that 
\begin{equation}\label{halfvolume}
8\pi nc\leq \varepsilon/2.
\end{equation} Once $c$ is given, we choose $\sigma_c$ and consequently $r_ c$ and $\int_{F_{ c}}(d\lambda)^n$ are fixed. Then we take $s>0$ such that 
$$s <\min \left\{s_1,\frac{r_ c}{2},
\left(\frac{ \varepsilon}{8\pi\int_{F_{ c}}(d\lambda)^n}
\right)^{\frac{1}{n}}\right\}.$$ 
Since $s<s_1$ and $s<\left(\frac{\varepsilon}{8\pi\int_{F_{ c}}(d\lambda)^n}
\right)^{\frac{1}{n}}$, we have
$$\int_{MT(F_ c,\psi)}\alpha_{s}\wedge (d\alpha_{s})^n\leq \varepsilon/2$$
and together with \eqref{volumebound} and \eqref{halfvolume} we get
$${\rm vol}(\alpha_{s, c})\leq\varepsilon.$$

We showed that for any $\varepsilon>0$ there exists a contact form $\alpha_{s,c}$ on $OB(F,\psi)$ that after a rescaling, fulfils the desired properties in Theorem \ref{maingss}. We recall that via the flow of the vector field $X$ on $V$, $V$ is identified with $OB(F,\psi)$. Abusing the notation, we denote the objects on $OB(F,\psi)$ that correspond to $\alpha$ and $h$ with the same letters respectively. Now we note that the properties of $\alpha_{s,c}$ that we established above are invariant under diffeomorphisms. Hence it would be enough to show that $\ker \alpha=\ker \alpha_{s,c}$ on $OB(F,\psi)$ in order to complete the proof of Proposition \ref{keyprop}. In fact again by the invariance  it is enough to show the following statement. 
\begin{prop}\label{supporting} There exists a diffeomorphism $\Psi:OB(F,\psi)\rightarrow OB(F,\psi)$ such that 
$$\Psi_*(\ker \alpha)=\ker \alpha_{s, c}.$$
\end{prop}
\begin{proof}
We note that the projection $MT(F,\psi)\rightarrow S^1$ induces an obvious open book 
$$\widetilde{\Theta}: OB(F,\psi)\setminus K \rightarrow S^1$$
on $OB(F,\psi)$. We want to show that this is a LOB with the contact binding form $\alpha_{s, c}$. We pick a defining function 
$$\widetilde{h}:OB(F,\psi)\rightarrow \mathbb{C}$$ as follows. Let 
$$\widetilde{u}:F\rightarrow [0,\infty)$$
be a smooth function such that for some suitably chosen $C>0$ and $\delta>0$ the followings hold. 
\begin{enumerate}[label=\textrm{(df\arabic*)}]
\item \label{df1} $\widetilde{u}(r,q)=r\left(1-\frac{r^2}{2r_c^2}\right)$ for
$r$ small enough.
\item \label{df2} $\widetilde{u}\equiv C$ on $([0,r_ c+\delta)\times K)^c$ and ${\rm supp}\,(\psi)\cap \imath_c([0,r_ c+\delta]\times K)=\emptyset$.
\item \label{df3} On $[0,r_ c+\delta]\times K$, $\widetilde{u}$ depends only on $r$ and $\frac{\partial\widetilde{u}}{\partial r}\geq 0$.
\end{enumerate}
Since $\widetilde{u}$ is constant on ${\rm supp}\,(\psi)$, the $S^1$-invariant extension of $\widetilde{u}$ is well-defined on $MT(F,\phi)$ and defines $|\widetilde{h}|$. We put $\widetilde{h}/|\widetilde{h}|:=\widetilde{\Theta}$. Observe that $r\widetilde{\Theta}$ and $1-r^2/2r_c^2$ are smooth functions on $r_c\mathbb{D}\times K$. Hence $\widetilde{h}$ descends to a smooth function on $OB(F,\psi)$. Note also that since $\frac{\partial\widetilde{u}}{\partial r}(0,q)=1$ for all $q\in K$, $\widetilde{h}$ is indeed a defining function, see \ref{h1}.  
\begin{lem} \textit{$d(\alpha_{s, c}/|\widetilde{h}|)$ induces an ILS on each fibre of $\widetilde{\Theta}$.}
\end{lem}
\begin{proof} We put
\begin{eqnarray}\label{lambdatilde}
\widetilde{\lambda}_x:=(\alpha_{s, c}/|\widetilde{h}|)_{|T(\{x\}\times F^\circ)}
\end{eqnarray} 
where $\{x\}\times F^\circ=\widetilde{\Theta}^{-1}(x)$. We check our claim on pieces of $F^\circ$ as follows. 
\begin{itemize}
\item \underline{On $\{x\}\times (0,r_ c]\times K$}: by \ref{df1} we have 
\begin{eqnarray}\label{lambdatildebeforer_eps}
\widetilde{\lambda}_x=\frac{f_{s,c}}{\widetilde{u}}\sigma_{ c}.
\end{eqnarray}
Observe that for given a function $\rho(r)$ on $(0,r_c]$, $d(\rho\,\sigma_c)$ is a positive symplectic form if $\rho>0$ and $\rho'<0$. In fact
$$(d(\rho\,\sigma_c))^n=(\rho'\,dr\wedge \sigma_c+\rho\,d\sigma_c)^n=(n-1)\,\rho^{n-1}\rho'\,dr\wedge \sigma_c\wedge (d\sigma_c)^{n-1}$$
and due to the parametrization (\ref{lambdanearK}), 
$dr\wedge \sigma_{c}\wedge(d\sigma_{c})^{n-1}$ is a negative volume form.

Now we note that $f_{s,c}/\widetilde{u}>0$ and 
$$\left(\frac{f_{s,c}}{\widetilde{u}}\right)'=\frac{f_{s,c}'\widetilde{u}-f_{s,c}\widetilde{u}_r}{\widetilde{u}^2}<0$$
since $f_{s,c}'\widetilde{u}-f_{s,c}\widetilde{u}_r<0$ due to \ref{f'} and \ref{df3}.
\newline
\item \underline{On $\{x\}\times [r_c,r_c+\delta]\times K$}: Note that by \ref{df2}, $\psi={\rm id}$ on this set. Hence we have
\begin{eqnarray}\label{lambdatildeafterr_eps}
\widetilde{\lambda}_x=\frac{s}{r\widetilde{u}}\sigma_{c}.
\end{eqnarray}
By \ref{df3}, $\widetilde{u}+r\widetilde{u}_r>0$ and therefore
$$\left(\frac{s}{r\widetilde{u}}\right)'=-s\,\frac{\widetilde{u}+r\widetilde{u}_r}{r^2\widetilde{u}^2}<0.$$
The claim follows as in the previous case.
\newline
\item \underline{On $\{x\}\times ([0,r_c+\delta)\times K)^c$}: By \ref{df2}  
\begin{eqnarray}\label{lambdatildefaraway}
\widetilde{\lambda}_x=\frac{s}{C}\left(\lambda+\kappa(x)\lambda_\psi\right)\;\Rightarrow\;d\widetilde{\lambda}_x=\frac{s}{C}d\lambda
\end{eqnarray}
which is clearly symplectic.
\end{itemize} 
\end{proof}
From the proof of the above lemma we also see that the ILSs 
$$\{d(\alpha_{s,\varepsilon}/|\widetilde{h}|)_{|T(\{x\}\times F^\circ)}\}_{x\in S^1}$$ are invariant under the flow of $\partial_x$. In fact $\partial_x$ is a symplectically spinning vector field for the LOB 
\begin{eqnarray}\label{LOB1}
\left(K,\widetilde{\Theta},\{d(\alpha_{s,c}/|\widetilde{h}|)_{|T(\{x\}\times F^\circ)}\}_{x\in S^1}\right).
\end{eqnarray}
Now we want to relate this LOB to the initial LOB on $V$, which now reads as
\begin{eqnarray}\label{LOB2}
\left(K,\widetilde{\Theta},\left\{d(\alpha/|h|)_{|T(\{x\}\times F^\circ)}\right\}_{x\in S^1}\right).
\end{eqnarray}
We also note that $\partial_x$ the symplectically spinning vector field for the open book \eqref{LOB2}, which corresponds to $X$ on $V$. In particular the ILSs
$$\left\{d(\alpha/|h|)_{|T(\{x\}\times F^\circ)}\right\}_{x\in S^1}$$
are invariant under the flow of $\partial_x$.
\begin{lem} There exists a diffeomorphism 
\begin{eqnarray}\label{Phi}
\Phi: OB(F,\psi)\rightarrow OB(F,\psi)
\end{eqnarray}
such that $\Phi\circ \widetilde{\Theta}=\widetilde{\Theta}\circ \Phi$ and the restriction of $\Phi$ to each fibre is symplectic, that is, for all $x\in S^1$,
$$\Phi^*d(\alpha_{s,c}/|\widetilde{h}|)_{|T(\{x\}\times F^\circ)}=d(\alpha/|h|)_{T(\{x\}\times F^\circ)}.$$
\end{lem}
\begin{proof}  
We have the following symplectic forms on $F^\circ$:
\begin{equation}\label{omegatilde}
\widetilde{\omega}:=d(\alpha_{s,c}/|\widetilde{h}|)_{|T(\{0\}\times F^\circ)},
\end{equation}
\begin{equation}\label{omega}
\omega:=d(\alpha/|h|)_{|T(\{0\}\times F^\circ)}.
\end{equation}
We want to show that
$$\omega_t:=(1-t)\omega+t\widetilde{\omega},\; t\in [0,1]$$
is a path of symplectic forms on $F^\circ$. 
Using the primitive of $\omega$ given by (\ref{lambda}) and  the primitive of $\widetilde{\omega}$ given by (\ref{lambdatilde}), we check our claim on the separate pieces of $F^\circ$. Let $t\in[0,1]$ be fixed. 
\begin{itemize}
\item \underline{On $\{x\}\times (0,r_c]\times K$}: By \eqref{lambdanearK} and \eqref{lambdatildebeforer_eps} we have 
$$\omega_t=d\left((1-t)\,\frac{1}{r}\sigma_c+t\,\frac{f_{s,c}}{\widetilde{u}}\sigma_c\right)=d\left(\rho\,\sigma_c\right)$$
where 
$$\rho:=(1-t)\,\frac{1}{r}+t\,\frac{f_{s,c}}{\widetilde{u}}.$$ 
We note that the functions $1/r$ and $f_{s,c}/\widetilde{u}$ are both positive and have strictly negative derivatives, see the proof of previous lemma. It follows that $\rho>0$ and $\rho'<0$. Hence $d(\rho\,\sigma_c)$ is a positive symplectic form.
\newline
\item \underline{On $\{x\}\times [r_c,r_c+\delta]\times K$}: By \eqref{lambdanearK} and \eqref{lambdatildeafterr_eps}  we get $\omega_t=\left(\rho(r)\,\sigma_c\right)$ where $$\rho=d(1-t)\frac{1}{r}\sigma_c+t\frac{s}{r\widetilde{u}}\sigma_c.$$
The claim follows as before.
\newline
\item \underline{On $\{x\}\times ([0,r_c+\delta)\times K)^c$}: By (\ref{lambdatildefaraway}) we have
$$\omega_t=(1-t)d\lambda+ t\frac{s}{C}d\lambda=\left((1-t)+\frac{ts}{C}\right)d\lambda.$$
\end{itemize}
Hence $\omega_t$ is symplectic on $F^\circ$ for all $t\in [0,1]$.

Reparametrizing the interpolation interval $[0,1]$ as $[0,2\pi]$ leads to a  smooth path of symplectic structures $\{\omega_t\}_{t \in [0,2\pi]}$ on $F^\circ$ such that $\omega_0 = \omega$ and $\omega_{2\pi} = \widetilde \omega$. Moreover by (\ref{lambdanearK}), \ref{f2} and \ref{df1}
$$\omega=\widetilde{\omega}=d\left(\frac{1}{r}\sigma_c\right)=\omega_t$$
near $K$ for all $t\in[0,2\pi]$. Then after applying the standard Moser argument we get a smooth isotopy 
$\{\psi_t\}_{t\in [0,2\pi]}$ of $F$ such that
\smallskip
\begin{enumerate}[label=\text{($\Psi$\arabic*)}]
\item   \label{psi1}   $\psi_0 = {\rm id}$; 

\smallskip
\item   \label{psi2}   $\psi_t={\rm id}$ near $K$ for all $t \in [0,2\pi]$; 

\smallskip
\item   \label{psi3}   $\psi_t^* \, \omega_t = \omega_0=\omega$ for all $t \in [0,2\pi]$.
\end{enumerate}
Now we define $\Phi : [0,2\pi] \times F \rightarrow [0,2\pi] \times F$ by
\begin{eqnarray}\label{isotopy}
\Phi(x,p) := \left( x,\psi_{2\pi} \circ \psi^{-1}_x \circ \psi^{-1} \circ \psi_x(p) \right)
\end{eqnarray}
where $\psi$ is the the monodromy given by \eqref{monodromy}. We note that 
$$
\Phi(2\pi,p) = \left(2\pi, \psi^{-1} \circ \psi_{2\pi}(p) \right),
$$
and by \ref{psi1}, 
$$
\Phi(0,\psi(p)) 
= \left(0, \psi_{2\pi}(p) \right) = \left( 0, \psi(\psi^{-1} \circ \psi_{2\pi}(p)) \right).
$$
Hence $\Phi$ descends to a diffeomorphism on $MT(F,\psi)$. Since $\psi = {\rm id}$ near $K$
and $\psi_x = {\rm id}$ near $K$ for each~$x$ by~\ref{psi2}, we have that $\Phi = {\rm id}$ on a neighbourhood of $\partial MT(F,\psi)$. 
Hence $\Phi$ descends to a diffeomorphism on~$OB(F,\psi)$ and by definition, $\Phi$ commutes with~$\widetilde {\Theta}$. 

We recall that $\psi^* \omega = \omega$ and by construction we also have $\psi^*\widetilde \omega=\widetilde\omega$. Therefore, $\psi^*\omega_t = \omega_t$ for all~$t \in [0,2\pi]$. 

As noted above $\partial_x$ is a symplectically spinning vector field for both LOBs~\eqref{LOB1} and~\eqref{LOB2}. Therefore once $\{x\} \times F^\circ$ is identified with $F^\circ=\{0\} \times F^\circ$ via the flow
of~$\partial_x$ then $d(\alpha_{s,c}/|\widetilde{h}|)_{|T(\{x\}\times F^\circ)}$ is identified with $\widetilde{\omega}$ and $d(\alpha/|h|)_{|T(\{x\}\times F^\circ)}$ is identified with $\omega$. Using these identifications and combining \eqref{isotopy} with \ref{psi3} we get 
\begin{eqnarray*}
\Phi^* d \bigl( \alpha_{s,c}/|\widetilde{h}| \bigr) |_{T (\{x\} \times F^\circ)}
&=&\big(\psi_{2\pi}\circ\psi^{-1}_x\circ \psi^{-1}\circ \psi_x\big)^* \widetilde \omega \\
&=&\psi_x^* \, (\psi^{-1})^* \, (\psi^{-1}_x)^* \, \psi_{2\pi}^* \omega_{2\pi} \\
&=&\psi_x^* \, (\psi^{-1})^* \, (\psi^{-1}_x)^* \, \omega_0  \\
&=&\psi_x^* \, (\psi^{-1})^* \omega_x  \\
&=&\psi_x^* \, \omega_x  \\
&=&\omega_0 \\
&=& \omega \\
&=& d(\alpha/|h|) |_{T (\{x\} \times F^\circ)} .
\end{eqnarray*} 
The proof of the lemma is complete.
\end{proof}
Now in order to finish the proof of Proposition \ref{supporting}, we fix the defining function $\Phi^*\widetilde{h}$ for the open book $(K,\widetilde{\Theta})$, which is indeed a defining function due to the construction of $\Phi$. Then as noted in the previous section, $|\Phi^*\widetilde{h}|/|h|$ is a positive function on  $OB(F,\psi)$. Consequently we have two contact forms 
$$\Phi^*\alpha_{s,c}\;\;{\rm and}\;\; \frac{|\Phi^*\widetilde{h}|}{|h|}\,\alpha$$ 
on $OB(F,\psi)$ such that 
$$d\left(\frac{\Phi^*\alpha_{s,c}}{|\Phi^*\widetilde{h}|}\right)=\Phi^*d(\alpha_{s,c}/|\widetilde{h}|)=d(\alpha/|h|)=d\left(\frac{\frac{|\Phi^*\widetilde{h}|}{|h|}\,\alpha}{|\Phi^*\widetilde{h}|}\right)$$
on $T (\{x\} \times F^\circ)$ for all $x\in S^1$. That is $\ker \Phi^*\alpha_{s,c} $ and $\ker \alpha=\ker \frac{|\Phi^*\widetilde{h}|}{|h|}\alpha$ are symplectically supported by the same LOB. Hence by Proposition \ref{prop21}, $\ker\Phi^*\alpha_{s,c}$ is isotopic to $\ker \alpha$. By Gray's stability theorem there is a diffeomorphism $\widetilde\Phi$ of $OB(F,\psi)$ such that $\widetilde\Phi_* (\ker \alpha) = \ker \Phi^* \alpha_{s,c}$. We set $\Psi := \Phi \circ\widetilde\Phi$ and get $$\Psi_* (\ker \alpha) = \ker \alpha_{s,c}.$$

\end{proof}

\end{document}